\documentclass[a4paper,12pt]{amsart}
\usepackage{hyperref}

\usepackage{microtype}
\usepackage{amssymb,amsmath,amsfonts,amsthm,bm,latexsym,amscd,verbatim,url,nicefrac,stmaryrd,enumerate,colonequals,dsfont,appendix}
\usepackage[all]{xy}

\usepackage{times}
\usepackage{graphicx}
\usepackage{fullpage}

{\bf}{\it}
\newtheorem*{thm*}{Theorem}
\newtheorem{thm}{Theorem}[section]{\bf}{\it}
\newtheorem{prop}[thm]{Proposition}

\newtheorem{cor}[thm]{Corollary}

\theoremstyle{definition}
\newtheorem{dfn}[thm]{Definition}
\theoremstyle{remark}
\newtheorem{rmk}[thm]{Remark}
\theoremstyle{remark}
\newtheorem{exm}[thm]{Example}
\newtheorem{assu}[thm]{Assumption}

\newcommand{\A}{\mathbb{A}}
\newcommand{\B}{\mathbb{B}}
\newcommand{\C}{\mathbb{C}}

\newcommand{\G}{\mathbb{G}}

\newcommand{\LL}{\mathbb{L}}

\newcommand{\Q}{\mathbb{Q}}
\newcommand{\RR}{\mathbb{R}}
\newcommand{\R}{\mathbb{R}}

\newcommand{\Z}{\mathbb{Z}}

\newcommand{\cat}{\mathbf{C}}
\newcommand{\catD}{\mathbf{D}}

\newcommand{\catT}{\mathbf{T}}

\newcommand{\ra}{\rightarrow}

\newcommand{\mcF}{\mathcal{F}}
\newcommand{\mcG}{\mathcal{G}}

\newcommand{\mcO}{\mathcal{O}}
\newcommand{\mcU}{\mathcal{U}}

\newcommand{\mfM}{\mathfrak{M}}

\newcommand{\mfX}{\mathfrak{X}}

\newcommand{\adj}[4]{#1\negmedspace: #2\rightleftarrows #3:\negmedspace #4}

\DeclareMathOperator{\an}{an}

\DeclareMathOperator{\Berk}{Berk}

\DeclareMathOperator{\coeq}{coeq}
\DeclareMathOperator{\colim}{colim}
\DeclareMathOperator{\cp}{cp}
\DeclareMathOperator{\cont}{cont}
\DeclareMathOperator{\ct}{ct}

\DeclareMathOperator{\eff}{eff}
\DeclareMathOperator{\et}{\acute{e}t}

\DeclareMathOperator{\Et}{Et}

\DeclareMathOperator{\qc}{qc}

\DeclareMathOperator{\FormDA}{\bf{FormDA}}

\DeclareMathOperator{\Frob}{Frob}
\DeclareMathOperator{\Frobet}{Frob\acute{e}t}

\DeclareMathOperator{\Gal}{Gal}

\DeclareMathOperator{\holim}{holim}
\DeclareMathOperator{\hocolim}{hocolim}
\DeclareMathOperator{\Hom}{Hom}

\DeclareMathOperator{\id}{id}

\DeclareMathOperator{\perf}{{perf}}
\DeclareMathOperator{\Perf}{Perf}

\DeclareMathOperator{\Rep}{Rep}

\DeclareMathOperator{\RigSm}{RigSm}

\DeclareMathOperator{\sep}{sep}

\DeclareMathOperator{\Sing}{Sing}
\DeclareMathOperator{\Sm}{Sm}

\DeclareMathOperator{\Spa}{Spa}
\DeclareMathOperator{\Spec}{Spec}

\DeclareMathOperator{\Ch}{\bf{Ch}}

\DeclareMathOperator{\DA}{\bf{DA}}
\DeclareMathOperator{\DM}{\bf{DM}}

\DeclareMathOperator{\Mod}{\bf{-Mod}}

\DeclareMathOperator{\AnDA}{\bf{AnDA}}

\DeclareMathOperator{\PerfDA}{\bf{PerfDA}}

\DeclareMathOperator{\RigDA}{\bf{RigDA}}
\DeclareMathOperator{\RigDM}{\bf{RigDM}}

\DeclareMathOperator{\RigSH}{\bf{RigSH}}
\DeclareMathOperator{\FormSH}{\bf{FormSH}}

\DeclareMathOperator{\Psh}{\bf{Psh}}

\DeclareMathOperator{\Set}{\bf{Set}}
\DeclareMathOperator{\sSet}{\bf{sSet}}
\DeclareMathOperator{\Sh}{\bf{Sh}}
\DeclareMathOperator{\SH}{\bf{SH}}

\DeclareMathOperator{\sPsh}{\bf{sPsh}}
\DeclareMathOperator{\Top}{\bf{Top}}
\DeclareMathOperator{\sTop}{\bf{sTop}}
\DeclareMathOperator{\wRRig}{\widehat{\bf{Rig}}}
\DeclareMathOperator{\wRigDA}{\bf{sPerfDA}}
\newcommand{\wwRigDA}{\wRRig\!\DA}

\DeclareMathOperator{\Ho}{Ho}

\begin{document}
	\title{The Berkovich realization for rigid analytic motives}
	\author{Alberto Vezzani}
	\address{LAGA - Universit\'e Paris 13 (UMR 7539)\\99 Av. Jean-Baptiste Cl\'ement\\93430 Villetaneuse, France}
	\email{vezzani@math.univ-paris13.fr}
	\thanks{The author was supported by the ANR grant PERCOLATOR ANR-14-CE25-0002, the ANR grant PERGAMO ANR-18-CE40-0017 and a PEPS JCJC grant from INSMI (CNRS)}

\begin{abstract}
 We prove that the functor associating to a rigid analytic variety the singular complex of the underlying Berkovich topological space is motivic, and defines the maximal Artin quotient of a motive. We use this to generalize Berkovich's results on the weight-zero part of the \'etale cohomology of a variety defined over a non-archimedean valued field.
\end{abstract}

\maketitle

\tableofcontents

\section{Introduction}

One of the key features of any motivic theory  over a field $k$ is the existence of realization functors, that is functors from the corresponding category of motives to some category of vector spaces (with further structure, if possible) that would produce and generalise regulator maps, comparison theorems and periods. In this paper, we will adopt the language of Morel-Voevodsky-Ayoub (mixed, \'etale, derived, with coefficients in $\Lambda$) motives  $\DA_{\et}(k,\Lambda)$. 
For an introduction to this language, we refer to \cite{ayoub-etale,ayoub-th2,ayoub-icm,cd,cd-etale,mvw}. We will mostly be interested to the case $\Lambda\supset\Q$ so that adding transfers (and hence considering the categories $\DM(k,\Lambda)$)  makes no difference in the theory (see \cite{ayoub-etale} and \cite{vezz-DADM}) under suitable hypotheses.
 
Whenever $k$ is a subfield of $\C$ one can consider the Betti realization (see \cite{ayoub-note}), the $\ell$-adic realizations \cite{ayoub-etale} or the de Rham realization \cite{ayoub-h1} (possibly enriched, see \cite{ivorra-h}). The well-known comparison theorems show that they are all equivalent, up to a change of coefficients. Among other things, these functors can be used to define motivic Galois groups \cite{ayoub-h1}, and some conjectural formal properties of them (say, being conservative on compact objects) reflect some deep geometrical facts of the theory of algebraic varieties (see \cite{ayoub-concon}). We remark that the Betti and the de Rham cohomologies (as vector spaces, with no extra structure) can be extended to, and defined by means of  the category of complex analytic motives $\AnDA(\C,\Lambda)$ (equivalent to $\catD(\Lambda)$, see \cite[Theorem 1.8]{ayoub-note}).

Whenever the characteristic of $k$ is positive, the array of possible realizations is more limited. There are $\ell$-adic realizations (but comparison theorems are not present in full generality) constructed in \cite{ayoub-etale}. For $p$-adic realizations, a natural approach would consist in associating to a variety (more generally, a motive) over $k$ a rigid analytic ``variety'' (better saying, a rigid analytic motive) over complete valued field $K$ of characteristic $0$ and residue equal to $k$, and then using realization functors for such objects. The problem is hence transferred into producing realization functors for rigid analytic motives $\RigDA_{\et}(K,\Lambda)$. In \cite{vezz-rigidreal} we constructed a de Rham-like realization, giving rise to the rigid realization on $\DA_{\et}(k,\Lambda)$. In \cite{bamb-vez} we construct the $\ell$-adic realizations compatible with the ones of $\DA_{\et}(k,\Lambda)$. In this article, we deal with a Betti-like realization functor. We will also show that it is \emph{not} of the same nature as the previous ones (in particular, one can not expect it to be conservative).

The most naive approach, in analogy to the complex Betti realization constructed in \cite{dug-isa}, consists in considering the singular homology of the  Berkovich space $|X(C)|_{\Berk}$ underlying the base change of rigid analytic variety $X/K$  to a complete algebraically closed field $C$. The first result of this paper is to show that this approach works, at least on the category of effective \'etale motives $\RigDA^{\eff}_{\et}(K,\Lambda)$ (see Theorem \ref{mainQ}).
\begin{thm}\label{mainQin}
 There is a triangulated functor
	$$
	{\LL B^*}\colon{\RigDA^{\eff}_{\et}(K,\Lambda)}\ra{\catD(\Lambda)}
	$$
	such that, for any rigid analytic variety $X$ and any $n\in\Z$ 
	$$H_n(\LL B^*\Lambda(X))\cong H_{n}^{\Sing}(|X(C)|_{\Berk},\Lambda).$$  
\end{thm}

On the other hand, this (co)homology theory is unsatisfying in many respects. Indeed, Berkovich spaces are ``too  contractible''. For example, $|\G_m(C)|_{\Berk}$ is (strongly) homotopically equivalent to a point, destroying therefore any information linked to monodromy and any hope to extend this realization to stable motives (obtained by formally inverting the Tate twist). On the other hand, some  results of Berkovich \cite{berk-tate} hint to the fact that this cohomology theory captures the weight-zero part of the other realizations, as he proves that for any algebraic variety $X$ over a discretely valued $K$ we have
\begin{equation}\label{w0B}
H^i_{\Sing}(|X(C)|_{\Berk},\Q_{\ell})\cong H^i_{\et}(X,\Q_{\ell})_0
\end{equation}
where the right hand side is the maximal sub-representation of $H^i_{\et}(X,\Q_\ell)$  on which (a lift of) Frobenius acts by roots of unity. The main result of this paper is to provide the following motivic interpretation/generalization of these formulas (see Theorem \ref{main0}).
\begin{thm}\label{main0in}
	Let $K$ be a complete valued field and $\Lambda$ be a $\Q$-algebra. The functor $\LL B^*$ can be enriched with a Galois action, so that $\LL B^*M$ is an Artin motive. Also, for any motive $M\in\RigDA^{\eff}_{\et}(K,\Lambda)$ there exists a canonical map 
	$$
	M\ra \LL B^* M
	$$
	which is universal  among maps from $M$ to an Artin motive. 
\end{thm}

 We point out that the result above implies not only the \emph{existence} of  a universal motivic Artin quotient (in the algebraic setting this is proved in \cite[Corollary 2.3.3]{ayoub-bv}) but also an explicit description of it in terms of Berkovich spaces: as an application, we can  answer positively to a conjecture of Ivorra and Sebag \cite{ivo-seb} and generalize Berkovich's formula \eqref{w0B}  to analytic varieties  (see Corollary \ref{maincor}) as follows.

\begin{cor}
Let $X$ be a quasi-compact rigid analytic variety (or more generally, a compact rigid analytic motive) over a  non-archimedean field $K$ with a 
finite residue field. We have the following isomorphism:
$$
H^i_{\Sing}(|X(C)|_{\Berk},\Q_{\ell})\cong H^i_{\et}(X,\Q_{\ell})_0.
$$
\end{cor}

In Section \ref{sbr} we prove Theorem \ref{mainQin} also in its simplicial variant (without coefficients) and in Section \ref{artq} we prove Theorem \ref{main0in}. In Section \ref{tilt} we show that the previous results are compatible with the motivic tilting equivalence of \cite{vezz-fw} defined whenever $K$ is perfectoid, while in Section \ref{etale} we deduce the formulas \eqref{w0B} via the \'etale realization functors.

\section{The  Berkovich realization}\label{sbr}

From now on, we consider a fixed base valued field $K$ as follows.

\begin{assu}\label{assu}
We let $K$ be a field which is complete with respect to a  non-trivial multiplicative valuation (of rank 1) $||\cdot||\colon K\ra \R_{\geq0}$.
\end{assu} 

The aim of this first section is to define a functor from the category of additive \'etale motives of rigid analytic varieties $\RigDA^{\eff}_{\et}(K,\Lambda)$ to the derived category of $\Lambda$-modules, such that the complex associated to the motive of a variety $X$ computes the singular (co-)homology of the Berkovich space $|X|_{\Berk}$ with coefficients in $\Lambda$. %

In order to define our functor, we will simply use the universal property of the categories of (effective, without transfers) motives, which we will now briefly recall in a more general setting for the convenience of the reader. All the results appear in \cite{dugger} (for the simplicial case) and  in \cite{gall-chou} (for the  case of complexes of presheaves) and we refer to these sources for definitions and proofs.

\begin{dfn}
Let  $\cat$ be any small category. We can endow the category $\sPsh(\cat)$ of simplicial presheaves on $\cat$ [resp the category $\Ch\Psh(\cat,\Lambda)$ of complexes of  presheaves of $\Lambda$-modules on $\cat$]  with the projective model structure, for which cofibrations and weak equivalences are defined point-wise. This defines a model category $\mcU\cat $ [resp. a $\Lambda$-enriched model category  $\mcU_{dg}\cat=\mcU_{\Ch(\Lambda)}\cat$]. 
\end{dfn}

The Yoneda embedding $\cat\ra\Psh(\cat)$ can be composed with the functor $\Set\ra\sSet$ sending any set to the constant simplicial set [resp. the functor $\Set\ra\Lambda\Mod$ sending a set to the free $\Lambda$-module attached to it]. This defines a Yoneda-like embedding $y\colon\cat\ra\mcU\cat$ [resp. $y\colon\cat\ra\mcU_{dg}\cat$]   which is universal in the following sense.

\begin{prop}[\cite{gall-chou,dugger}]\label{martin}
	Let $\gamma\colon\cat\ra\catD$ be any functor, and suppose $\catD$ is endowed with a [$\Lambda$-enriched] model category structure. There exists a Quillen functor  $L\colon \mcU\cat\ra\catD$ [resp. a Quillen functor $L\colon \mcU_{dg}\cat\ra\catD$ of $\Lambda$-enriched model categories] such that the induced triangle 
	$$\xymatrix{
				\cat\ar[r]^{\gamma}\ar[d]_{y}&\catD\\
				\mcU\cat\ar[ur]_{L}		
			}
		$$
	is commutative up to a weak equivalence $L\circ y\Rightarrow\gamma$. Moreover $L$ is unique up to a contractible choice.
\end{prop}

Suppose now that $\cat$ is endowed with a Grothendieck topology $\tau$ and a choice of an object $I$. Under some hypotheses, we can consider the Bousfield localization of $\mcU\cat $ and $\mcU_{dg}\cat $ with respect to $\tau$-hypercovers and $I$-homotopy, in the following sense.

\begin{dfn}\label{loc}
Let $T$ be a dense set of $\tau$-hypercovers. Consider the set of arrows $S$ in $\mcU\cat$ [resp. $\mcU_{dg}\cat$] given by $$\begin{aligned}
S=&\{\hocolim h(\mcU_\bullet)[i]\ra h(X)[i]\colon (\mcU_\bullet\ra X)\in T, i\in\Z\}\cup\\
&\{h(X\times I)[i]\ra h(X)[i]\colon X\in\cat, i\in\Z\}.
\end{aligned}$$ The [dg-enriched] Bousfield localization of $\mcU\cat$ with respect to $S$ will be denoted it by $\mcU\cat/(\tau,I)$ [resp. $\mcU_{dg}\cat/(\tau,I)$]. 
\end{dfn}

The model categories above  still enjoy a universal property, by composing Proposition \ref{martin} with the universal property of localizations.

\begin{prop} [{\cite[Corollary 5.14]{gall-chou}}]\label{martin2}
	Let $\gamma\colon\cat\ra\catD$ be any functor, and suppose $\catD$ is endowed with a [$\Lambda$-enriched] model category structure. The Quillen functor $L$ of Proposition \ref{martin} factors over $ \mcU\cat/(\tau,I)$ [resp. $\mcU_{dg}\cat/(\tau,I)$] whenever  $\gamma(X\times I)\ra\gamma(X)$ is a weak equivalence  and $\hocolim\gamma(\mcU_\bullet)\cong\gamma(X)$ for each $\tau$-hypercover $\mcU_\bullet\ra X$ in $\cat$. 
\end{prop}

\begin{rmk}
Thanks to its universal property, the construction of $\mcU(\cat)/(\tau,I)$ is functorial on the triples $(\cat,\tau,I)$ in some suitable sense.
\end{rmk}

\begin{rmk}\label{rmk:dense}
	One may omit the choice of the object $I$ and consider only the localization with respect to $\tau$. In this case, the  category $\mcU(\cat)/(\tau) $ [resp. $\mcU_{dg}(\cat)/(\tau) $] is Quillen  equivalent to the categories of simplicial [resp. complexes of] sheaves over $\cat$ endowed with its local model structure (see \cite[Corollary 4.4.42]{ayoub-th2}). In particular, one can replace $\cat$ by a $\tau$-dense full subcategory without changing the homotopy category. We remark  that the homotopy category $\Ho\mcU_{dg}(\cat)/(\tau) $ is then equivalent to the (unbounded) derived category $\catD(\Sh_{\tau}(\cat,\Lambda))$ of sheaves of $\Lambda$-modules.
	\end{rmk}

\begin{exm}
	Suppose we take $\cat=\Sm/k$ the category of smooth varieties over a field $k$. We can endow it with the \'etale topology   and we can select $I$ to be the affine line $\A^1_k$. The homotopy category $\Ho(\mcU_{dg}(\Sm/k)/(\et,\A^1))) $  is the category of effective Voevodsky motives without transfers $\DA^{\eff}(k,\Lambda)$.
\end{exm}

\begin{exm}\label{ex:Gal}
	Consider $\cat$ resp. $\cat'$ to be the category of finite \'etale extensions resp. finite Galois extensions of a field $K$ and endow them with the \'etale topology. The two homotopy categories are canonically equivalent to $\catD\Sh(\Et/K,\Lambda)$ which we will denote by $\catD_{\et}(K,\Lambda)$ following \cite{ayoub-h2}.%
\end{exm}

We recall the following classical statement.

\begin{prop}[{\cite[Remark 1.21]{ayoub-h2}}]\label{sigma}
	Fix a separable closure $C$ of $K$. The category $\Sh_{\et}(K,\Lambda)$ is equivalent to the category of continuous $\Gal(C/K)$-representations by means of the functor $$\sigma^*\colon \mcF\mapsto\varinjlim_{ L\subset C, L/K \text{ finite Galois}}\mcF(L).$$
	In particular, the category %
	$\catD_{\et}(K,\Lambda)$ is equivalent to the derived category of the (semi-simple) category of %
	continuous $\Lambda$-representations of $\Gal(K^{\sep}/K)$.
\end{prop}

Along this article, we will adopt Huber's notations for rigid analytic varieties, fully faithfully embedded in the category of adic spaces (see \cite{huber}). For any Tate algebra $R$ we will write $\Spa R$ for the space $\Spa(R,R^\circ)$. For a rigid analytic variety $X$, the underlying topological space $|X|$ is a spectral space. It coincides with the sober topological space associated to the G-topos of Tate (see \cite[1.1.11]{huber}). It has a maximal Hausdorff quotient $|X|_{\Berk}$ which coincides with Berkovich's definition (see \cite{berkovich}) of the topological space of an analytic variety (\cite[Lemma 8.1.8 and Proposition 8.3.1]{huber}). We will use this last topological space to define our realization functor, as its properties are more akin to the classical complex situation.

\begin{exm}\label{ex:Rig}
	Suppose we take $\cat=\RigSm/K$ the category of smooth rigid analytic varieties over  $K$ (see \cite{BGR}). We can endow it with the \'etale topology  (defined in \cite{huber}) and we can select $I$ to be the closed disc $\B^1=\Spa K\langle T\rangle$. The homotopy category $\Ho(\mcU_{dg}(\RigSm/K)/(\et,\B^1)) $  is the category of rigid analytic Ayoub motives  (effective, \'etale, without transfers) $\RigDA^{\eff}_{\et}(K,\Lambda)$ (see \cite{ayoub-rig}). Whenever $K$ is perfect, we can also consider the $\Frobet$-topology, that is the one generated by \'etale covers and the relative Frobenius maps. In this case we obtain the category $\RigDA^{\eff}_{\Frobet}(K,\Lambda)$. By means of Remark \ref{rmk:dense}, in the construction we can replace  the category $\RigSm/K$ with its full subcategory of quasi-compact smooth varieties $\RigSm^{\qc}/K$ (or even, the one of smooth affinoid varieties) without changing the motivic category.
\end{exm}

\begin{rmk}\label{DADM}
We will mostly be interested in the case $\Q\subset\Lambda$. In this setting, the category of motives \emph{with transfers} $\RigDM^{\eff}(K,\Lambda)$ defined in \cite{ayoub-rig} is canonically equivalent to $\RigDA^{\eff}_{\Frobet}(K^{\Perf},\Lambda)$ where $K^{\Perf}$ is the completed perfection of $K$ (see \cite{vezz-DADM} and \cite[Proposition 5.20]{vezz-tilt4rigid}). One can therefore rephrase the main theorems of this article in terms of $\RigDM^{\eff}(K,\Lambda)$.
\end{rmk}

We now prove the existence of the simplicial version of the Berkovich realization. Here, we endow the category of topolgical spaces $\Top$ with the classical Quillen model structure, see \cite[Section 2.4]{hovey}.

\begin{prop}\label{maintop}	
The functor $X\mapsto|X|_{\Berk}$ induces a Quillen adjunction
	$$
	\mcU\Psh(\RigSm/K)/(\et,I)\rightleftarrows\Top.
	$$
	If $K$ is perfect, the adjunction above descends to 
	$$
\mcU\Psh(\RigSm/K)/(\Frobet,I)\rightleftarrows\Top.
$$
\end{prop}

\begin{proof}
	We first prove that $B^*$ sends the maps $\mcU_\bullet\ra X$ to weak equivalences, for any \'etale hypercover $\mcU_\bullet$ of $X$. To this aim, by   \cite[Theorem 8.6]{dhi} it suffices to show that  $|\bigsqcup_i{U_i}|_{\Berk}\cong\bigsqcup_i|U_i|_{\Berk}$ for any finite set of rigid varieties $\{U_i\}_{i\in I}$ (which is obvious)  and that $\LL B^*(\hocolim U_\bullet)\cong\LL B^*X$ where $\ U_\bullet$ is a \emph{split basal }\'etale hypercover of $X$. This means in particular that:
	\begin{enumerate}[(i)]
		\item $U_0$ is representable, and $U_0\ra X$ is \'etale surjective.
		\item\label{split} There exists a  representable presheaf $N_k$ for each $k$ such that $U_i=\bigsqcup_\sigma N_\sigma$ where $\sigma$ runs among surjections $[i]\ra[k]$ in the simplex category  $\mathbf{\Delta}$ with variable $k$, and $N_\sigma$ is a copy of $N_k$.
	\end{enumerate}
	We warn the reader that we constantly  abuse  notation by indicating with $U$ both a space and the presheaf it represents.

	As we already remarked, the functor $B^*$ preserves coproducts. In particular, the simplicial object $|U_\bullet|_{\Berk}$ is also a split simplicial topological space, that is, it enjoys property \eqref{split} with ``topological space" in place of ``representable presheaf".
	
	The homotopy colimit functor is the left derived Quillen functor of the $\colim$ functor, targeting the category  $\Top$, endowed with usual Quillen model structure,  from  the diagram category $\sTop$, endowed with the induced Reedy model structure.  By the ``topological trick" of Dugger-Isaksen \cite[Theorem A.7]{dug-isa} this homotopy colimit is weakly equivalent to the one computed with respect to the Str\o m model category structure on $\Top$ (for which all objects are fibrant and cofibrant, and weak equivalences are actual homotopy equivalences) and the induced  Reedy model structure on $\sTop$ . 
	
	Since $|U_\bullet|_{\Berk}$ is split, we deduce that it is Str\o m-cofibrant in $\sTop$.  In particular, its homotopy colimit coincides with its colimit, which is in turn isomorphic to 
	\[
	\coeq\left(	|U_1|_{\Berk}  \rightrightarrows |U_0|_{\Berk}\right)
	\]
	On the other hand, by definition of an \'etale hypercover, the map $|U_1|_{\Berk}\ra|U_0\times_XU_0|_{\Berk}$ is surjective (see \cite[Lemma 5.11]{berk-contr}). The same is true for the map $|U_0\times_XU_0|_{\Berk}\ra|U_0|_{\Berk}\times_{|X|_{\Berk}}|U_0|_{\Berk}$ (see \cite[Lemma 3.9(i)]{huber2}). We deduce that the maps of the diagram above factor over $E\colonequals |U_0|_{\Berk}\times_{|X|_{\Berk}}|U_0|_{\Berk}$. The map $U_0\ra X$ is an \'etale cover, and hence $|U_0|_{\Berk}\ra|X|_{\Berk}$ is a quotient map of topological spaces (see \cite[Lemma 5.11]{berk-contr}), with respect to the equivalence relation $E$. The coequalizer above is then simply $|U_0|_{\Berk}/E=|X|_{\Berk}$ as wanted.
	
	We now assume $K$ is perfect of positive characteristic and we prove that $B^*$ sends the relative Frobenius maps $X^{(1)}\ra X$ to weak equivalences. This follows at once since the relative Frobenius map induces actually a homeomorphism $|X^{(1)}|\cong|X|$. This implies that the functor $B^*$ factors over the $\Frobet$-localization.

	We now claim that the  left derived functor $\LL B^*$ sends the maps %
	$\pi_X\colon X\times \B^1\ra X$  to isomorphisms, i.e. that  $|X\times\B^1|_{\Berk}\ra|X|_{\Berk}$ is a weak equivalence in $\Top$. This follows from Berkovich's results about the contractibility of the disc \cite[Theorem 6.1.4]{berkovich} and \cite[Corollary 8.7(ii)]{berk-contr}. The result then follows from Proposition \ref{martin2}. %
\end{proof}

\begin{rmk}We denote  by $|X|$ the topological space underlying a rigid analytic variety (following Huber). 
	If $U\ra X$ is an \'etale cover, then the map  $|U|\ra|X|$ is open and surjective, and hence a quotient map. This shows that the functor $|\cdot|$ induces a Quillen functor from $\sPsh(\RigSm/K)$ to $\Top$ factoring over the \'etale localization. 
	On the other hand, %
	the topological space $|\B^1|$ is not weakly contractible,  and hence this functor does not factor over motivic category.%
\end{rmk}

\begin{rmk}
	Any Nisnevich square induces a cover of Huber spaces which admits locally a section (see \cite[Remark 1.2.4]{ayoub-rig}). The proof of the Nisnevich descent is therefore simpler, and it substantially coincides with the archimedean version of Proposition \ref{maintop} proven in \cite{dug-isa}.
\end{rmk}

\begin{thm}\label{mainQ}
Let $\Lambda$ be a ring. There is a Quillen adjunction
	$$
	\adj{B^*}{(\Ch\Psh(\RigSm/K,\Lambda))/(\et,\B^1)}{\Ch(\Lambda)}{B_*}
	$$
	inducing an adjunction:
	$$
	\adj{\LL B^*}{\RigDA^{\eff}_{\et}(K,\Lambda)}{\catD(\Lambda)}{\RR B_{*}}.
	$$
		If $K$ is perfect, the adjunction above descends to $\RigDA^{\eff}_{\Frobet}(K,\Lambda)$. 
	Moreover, for any rigid analytic variety $X$ and any $n\in\Z$ we have 
	$$\LL B^*\Lambda(X)\cong C_{\Sing}(|X|_{\Berk},\Lambda)$$  
	where $C_{\Sing}$ denotes the singular complex. In particular
	$$H_n(\LL B^*\Lambda(X))\cong H_{n}^{\Sing}(|X|_{\Berk},\Lambda).$$  

\end{thm}

\begin{proof} %
	For simplicity, we directly assume that $K$ is perfect. 
We consider the functor $B\colon \RigSm/K\ra\Ch(\Lambda)$ given by $X\mapsto C_{\Sing}(|X|_{\Berk},\Lambda)$. It induces a Quillen adjunction from $\mcU_{dg}\RigSm/K\ra\Ch(\Lambda)$ and we want to show that it factors over the $(\Frobet,\B^1)$-localization.

The functor ${\Sing}\colon\Top\ra\sSet$ mapping each topological space to its singular simplicial complex is an part of an exact Quillen equivalence of model categories. We then deduce from Proposition \ref{maintop} that the functor $\tilde{B}\colon X\mapsto {\Sing}(|X|_{\Berk})$ induces a Quillen adjunction $\mcU\RigSm/K\rightleftarrows\sSet$ factoring over the $(\Frobet,\B^1)$-localization. 

We now consider the left Quillen functor $N\Lambda\colon\sSet\ra\Ch(\Lambda)$ induced by the composition of the $\Lambda$-enrichment $\sSet\ra s\Lambda\Mod$ followed by the Dold-Kan functor $s\Lambda\Mod\ra\Ch(\Lambda)$. It gives rise to the following commutative diagram of left Quillen functors
	$$\xymatrix{
	\mcU\RigSm/K\ar[d]\ar[r]^-{\tilde{B}}&
\sSet\ar[d]^{N\Lambda}\\
	\mcU_{dg}\RigSm/K\ar[r] & \Ch(\Lambda)
	}$$
	Since the functor on top factors over the $(\Frobet,\B^1)$-localization, we deduce that the bottom functor does as well. On the other hand, we remark that this functor is the one  induced by mapping an object $X$  of $\RigSm/K$ to $N\Lambda(\Sing(|X|_{\Berk}))$ which is canonically isomorphic to $B(X)$ therefore proving our claim.
\end{proof}

\begin{rmk}
	By replacing simplicial sets with spectra in the construction of $\mcU\cat$ one can deduce from the previous result the existence of the Berkovich realization 
	$$\adj{\LL B^*}{\RigSH^{\eff}_{\et}(K)}{\SH}{\RR B_*}$$
	for the category of effective (anabelian) motives $\RigSH^{\eff}_{\et}(K)$ (see \cite[Definition 1.3.2]{ayoub-rig}) with values in the stable homotopy category $\SH$ (see \cite{EKMM}).
\end{rmk}

Having defined a realization from a category of motives, it is natural to see what sort of cohomology theory arises from it. As a matter of fact, this cohomology theory turns out to be quite pathological, as the two following remarks explain.

\begin{rmk}\label{tatetwist}
	The Quillen pair of Theorem \ref{mainQ} does not descend to the stable category of motives, the one constructed from $\RigDA_{\et}^{\eff}(K,\Lambda)$ by inverting the Tate twist in a universal way (see \cite[Chapter 2]{ayoub-rig}). Indeed, the object defining the Tate twist $[\Spa K\langle T^{\pm1}\rangle/*]$ is mapped to the zero complex since $*\ra|\Spa K\langle T^{\pm1}\rangle|_{\Berk}$ is a homotopy equivalence (see \cite{berk-contr}).
\end{rmk}

\begin{rmk}\label{resgen}Let $k$ be the residue field of $K$. 
	By means of \cite[Definition 1.4.12]{ayoub-rig} we can define the category $\FormSH^{\eff}_{\mfM}(K^\circ)$ of motives of formal schemes over $K^\circ$. The special fiber functor  and the generic fiber functor define triangulated functors (see \cite[Remark 1.4.25]{ayoub-rig})
$$
\DA^{\eff}_{\et}(k,\Lambda)\stackrel{\sim}{\leftarrow} \FormDA^{\eff}_{\et}(K^\circ,\Lambda)\ra\RigDA^{\eff}_{\et}(K,\Lambda)
$$
the first one being an equivalence by \cite{ayoub-rig}. By composition, we obtain in particular a cohomological realization 
$$
\DA^{\eff}(k,\Lambda)\stackrel{\sim}{\ra} \FormDA^{\eff}_{\et}(K^\circ,\Lambda)\ra\RigDA^{\eff}_{\et}(K,\Lambda)\stackrel{\LL B^*}{\longrightarrow}\catD(\Lambda)
$$
which is surprising at first sight: it looks as if it defined a cohomology with $\Z$-coefficients for varieties in positive characteristic (by taking $\Lambda=\Z$). This is not quite the case. Indeed, since rigid analytic varieties with good reduction are contractible (see \cite{berk-contr}), the composite realization above coincides with the one induced by the functor mapping a connected smooth variety $\bar{X}$ to the 
 trivial topological space. In particular, the homology theory on connected smooth algebraic varieties over $k$ obtained through the composite realization  is %
\[
H_*(\bar{X})=\begin{cases}
\Lambda&\text{if }{i=0}\\
0&\text{if }{i>0}
\end{cases}
\] 
\end{rmk}

\section{The Berkovich realization as the maximal Artin quotient}\label{artq}

From now on, we make the following assumption:

\begin{assu}
We suppose that $\Lambda$ is a $\Q$-algebra. Fix a separable closure $K^{\sep}$ of $K$ and let $C$ be its completion.
\end{assu} 

The first aim of this section is to enrich the realization constructed in Theorem \ref{mainQ} into a functor taking values in Galois representations. 

We recall some crucial results of Berkovich on the singular cohomology of Berkovich spaces that we list below.

\begin{prop}\label{berkH}
Let $X$ be a smooth quasi-compact rigid analytic variety over $K$.
\begin{enumerate}[(i)]
\item\label{b1} $H^i_{\Sing}(|X|_{\Berk},\Lambda)\cong H^i(|X|_{\Berk},\Lambda)$ and they have finite dimension.
\item There is a finite Galois extension $L$ such that $H^i(|X_L|_{\Berk},\Lambda)\cong H^i(|X_{L'}|_{\Berk},\Lambda)$
for each field extension $L'/L$.
\item If $L/K$ is a finite Galois extension then  $|X|_{\Berk}\cong \Gal(L/K)\backslash |X_L|_{\Berk}$ and $H^i(|X|_{\Berk},\Lambda)\cong H^i(|X_L|_{\Berk},\Lambda)^{\Gal(L/K)}$.
\item\label{4} $|X|_{\Berk}\cong \Gal(K^{\sep}/K)\backslash |X_C|_{\Berk}$ and  $\Gal(K^{\sep}/K)$ acts continuously on the $\Lambda$-module $ H^i(|X_C|_{\Berk},\Lambda)$.
\item If $X$ is connected and has good reduction, then $|X|_{\Berk}$ is contractible. In particular $H^i_{\Sing}(|X|_{\Berk},\Lambda)=0$ if $i>0$.
\end{enumerate}
\end{prop}

\begin{proof}
The first statement follows from {\cite[Corollary 9.6]{berk-contr}}. If $X$ is smooth, it is locally isomorphic to $\Spa K\langle\tau_1,\ldots,\tau_n\rangle/(p_1,\ldots,p_n)$ with $p_i$ polynomials in $K[\tau_i]$. In particular, it is an open subvariety of the analytification of $\Spec [\tau_1,\ldots,\tau_n]/(p_1,\ldots,p_n)$. We can then apply {\cite[Theorem 10.1]{berk-contr}} to get the second point. The third point follows from \cite[Proposition 1.3.5]{berkovich} and \cite[Paragraph 5.3]{tohoku} while the fourth statement follows from \cite[Corollary 1.3.6]{berkovich} and the previous points. The fifth point is proved in \cite[Section 5]{berk-contr} (the skeleton of a rigid analytic variety of good reduction is a point).
\end{proof}

\begin{dfn}\label{defBGal}We let $\Gal/K$ be the category of finite Galois extensions of $K$ inside $K^{\sep}$. 
For any smooth quasi-compact  variety $X$ and any $F$ in $\Gal/K$, we remark that the   complex of $\Lambda$-modules
$C_{\Sing}(|X_F|_{\Berk},\Lambda)
$
comes equipped with a canonical continuous action of $\Gal(K^{\sep}/K)$. Since $\Q\subset\Lambda$, it is a complex of acyclic Galois representations. We denote with $B_{F}(X)$  the induced object of $\Ch\Sh_{\et}(\Gal/K,\Lambda)$:
$$
B_{F}(X)\colon \Spec L\mapsto  C_{\Sing}(|X_F|_{\Berk},\Lambda)^{\Gal(K^{\sep}/L)}.
$$
We finally define $B_{\Gal(K)}(X)$ to be  $\holim_{F\in\Gal/K}B_F(X)$ in $\Ch\Psh(\Gal/K,\Lambda)$. We denote with the same symbol the corresponding object in $\Ch\Sh_{\et}(\Gal/K,\Lambda)\cong \Ch\Sh_{\et}(\Et/K,\Lambda)$ (see Example \ref{ex:Gal}).
\end{dfn}

 \begin{rmk}
	In the defintion of $B_{\Gal(K)}(X)$ we use a homotopy limit over the complexes obtained with finite Galois extensions rather than taking the singular complex of $|X_C|$. This is akin to the situation considered by Quick \cite[Section 3]{quick}.
\end{rmk}

\begin{rmk}\label{holimpsh}
	The object $B_{\Gal(K)}(X)$  is a homotopy limit of the \'etale-fibrant complexes $B_F(X)$ (they are levelwise Galois-acyclic) hence it is also \'etale fibrant. We deduce that it can also be computed directly in $\Ch\Sh_{\et}(\Gal/K,\Lambda)$ as a homotopy limit of the complexes of sheaves $B_F(X)$.
\end{rmk}

\begin{rmk}\label{evL}
	Let $L$ be in $\Gal/K$. The functor $ev_L\colon\mcF\mapsto\mcF(L)$ from $\Ch\Psh(\Gal/K,\Lambda)$ to $\Ch(\Lambda\Mod)$ is exact and preserves homotopy limits. Moreover, the collection of functors $\{ ev_L\}_{L\in\Gal/K}$ reflects the weak equivalences (that is, a map $\mcF\ra\mcG$ is a weak equivalence if and only if all maps $\mcF(L)\ra\mcG(L)$ are quasi-isomorphisms). This follows from the very definition of the projective model structure that we put on $\Ch\Psh(\Gal/K,\Lambda)$.
\end{rmk}

\begin{prop}\label{isXC}
The sheaf $B_{\Gal(K)}(X)$ corresponds to the continuous Galois representations $H_i^{\Sing}(|X_C|_{\Berk},\Lambda)$ by means of the equivalence given in Proposition \ref{sigma}.
\end{prop}

\begin{proof}
 By construction we have $$B_{\Gal(K)}(X)( L)\!=\!(\holim_FB_F(X))(L)\cong \holim_F(B_F(X)( L))\!\cong\!\holim_F(C_{\Sing}(|X_F|)^{\Gal(F/L)})$$
whose homology is by \cite[Theorem 3.15(1) and 3.15(6)]{prasolov} and by Proposition \ref{berkH} isomorphic to  $\varprojlim_F H_i(|X_F|)^{\Gal(K^{\sep}/L)}\cong H_i(|X_L|)\cong H_i(|X_C|)^{\Gal(K^{\sep}/L)}$ (we again used the fact that $\Q\subset\Lambda$).
\end{proof}

We are now ready to enrich the Berkovich realization with a Galois action. We recall that by Example \ref{ex:Rig} the motivic category $\RigDA_{\et}^{\eff}(K,\Lambda)$  can be equivalently defined out of the full subcategory $\RigSm^{\qc}/K$ of $\RigSm/K$ whose objects are \emph{quasi-compact} smooth varieties.

\begin{prop}
The functor $B_{\Gal(K)}\colon \RigSm^{\qc}/K\ra\Ch\Sh_{\et}(K,\Lambda)$ induces a Quillen adjunction
	$$
	\adj{B_{\Gal(K)}^*}{(\Ch\Psh(\RigSm/K,\Lambda))/(\et,\B^1)}{\Ch\Sh_{\et}(K,\Lambda)}{B_{\Gal(K)*}} %
	$$
	and hence an adjunction:
	$$
	\adj{\LL B^*_{\Gal(K)}}{\RigDA^{\eff}_{\et}(K,\Lambda)}{\catD_{\et}(K,\Lambda)}{\RR B_{\Gal(K)*}}.
	$$
	If $K$ is perfect, the adjunction above descends to $\RigDA^{\eff}_{\Frobet}(K,\Lambda)$.
\end{prop}

\begin{proof}
By Proposition \ref{martin2}, it suffices to prove that the functor $$B^*_{\Gal(K)}\colon \mcU_{dg}\RigSm^{\qc}/K\ra\Ch\Psh(\Gal/K,\Lambda)\cong\mcU_{dg}(\Gal/K)$$ sends the maps in the set $S$ of Definition \ref{loc} to weak equivalences. By Remark \ref{evL} we can fix a Galois extension $L/K$ and check that the composite functor
$$\mcU_{dg}\RigSm^{\qc}/K\ra\Ch\Psh(\Gal/K,\Lambda) %
\stackrel{ev_L}{\ra} \Ch(\Lambda)$$
sends the maps in the set $S$ to weak equivalences. On the other hand, we remark that by definition the functor above is the one induced by $X\mapsto \holim_F C_{\Sing}(|X|_F,\Lambda)^{\Gal(K^{\sep}/L)}$. This last complex is canonically quasi-isomorphic to $C_{\Sing}(|X|_L,\Lambda)$ as the following sequence of isomorphisms shows, where we let for simplicity $F\supset L$ (we repeatedly use the hypothesis $\Q\subset\Lambda$ and Proposition \ref{berkH}):
$$
\begin{aligned}
H_i(C_{\Sing}(|X_F|,\Lambda)^{\Gal(K^{\sep}/L)})&\cong H_i^{\Sing}(|X_F|,\Lambda)^{\Gal(K^{\sep}/L)}\cong H_{\Sing}^i(|X_F|,\Lambda)^{\vee\Gal(K^{\sep}/L)}\\
&\cong 
H_{\Sing}^i(|X_F|,\Lambda)^{\Gal(K^{\sep}/L)\vee}\cong
H_{\Sing}^i(|X_L|,\Lambda)^{\vee}\\&\cong H^{\Sing}_i(|X_L|,\Lambda)\cong H_i(C_{\Sing}(|X_L|,\Lambda)).
\end{aligned}
$$
 We already proved that the functor $X\mapsto C_{\Sing}(|X_L|,\Lambda)$ factors over the $(\Frobet,\B^1)$-localization in Theorem \ref{mainQ}  hence the set $S$ is sent to weak equivalences, as claimed.
\end{proof}

We recall that there is another adjunction between the categories above but defined in the opposite direction:  it is the pair induced by the inclusion of the small site into the big site $\iota\colon  \Et/K\ra\RigSm/K$ giving rise to:
$$\adj{\LL \iota^*}{\catD_{\et}(K,\Lambda)}{\RigDA^{\eff}_{\et}(K,\Lambda)}{\RR \iota_*}
$$

\begin{dfn}
The objects in the essential image of $\LL\iota^*$ are called \emph{Artin motives}, and the full subcategory they form is denoted  by $\RigDA^{\eff}_{\et}(K,\Lambda)_0$ (or $\RigDA^{\eff}_{\Frobet}(K,\Lambda)_0$ ).
\end{dfn}

\begin{thm}\label{main0}
Let 
 $\Lambda$ be a $\Q$-algebra. The inclusion of Artin motives in effective rigid analytic motives over $K$ $$ \RigDA^{\eff}_{\et}(K,\Lambda)_0\subset \RigDA^{\eff}_{\et}(K,\Lambda)$$
admits a left adjoint $\omega_0\colonequals\LL\iota^*\circ\LL B^*_{\Gal(K)}$. 
In particular, for any motive $M$ the map 
$$
M\ra \omega_0 M
$$
 is universal  among maps from $M$ to an Artin motive. If $K$ is perfect, the same is true with respect to the category  $\RigDA^{\eff}_{\Frobet}(K,\Lambda)$
\end{thm}

In other words, we want to prove the following result.

\begin{thm}\label{mainadj}
Let 
$\Lambda$ be a $\Q$-algebra. The functor $\LL B^*_{\Gal(K)}$ on \'etale motives (or $\Frobet$-motives if $K$ is perfect) is a left adjoint to the functor $\LL\iota^*$ and the unit map $\LL B^*_{\Gal(K)}\LL\iota^*\Rightarrow\id$ is invertible.
\end{thm}

\begin{proof}[Proof of Theorem \ref{main0} from Theorem \ref{mainadj}]
By definition, the category $ \RigDA^{\eff}_{\et}(K,\Lambda)_0$ is the essential image of $\LL\iota^*$ which is fully faithful given that $\LL B^*_{\Gal(K)}\LL\iota^*\cong\id$. The two categories are then equivalent, and the adjunction pair of Theorem \ref{main0} can be deduced from the one of Theorem \ref{mainadj}.
\end{proof}

\begin{rmk}
The content of the previous results does not lie in the existence of a left adjoint functor $\omega_0$ which could be proved with purely categorical methods (see  \cite[Chapter 5]{BVK}) but rather in its explicit description through  Berkovich spaces. This  produces interesting applications, see Section \ref{etale}.
\end{rmk}

We prove Theorem \ref{mainadj} in several steps. We start by checking the last claim.

\begin{prop}\label{unit}
There is an invertible natural transformation  $\LL B^*_{\Gal(K)}\circ\LL\iota^*\cong\id$. %
\end{prop}

\begin{proof}
Let $L$ be a fixed finite Galois extension of $K$. The object $\LL(B_{\Gal (K)}\circ\iota)^*(\Lambda(L))$ is the following complex:
$$
\holim_FC_{\Sing}(|\Spa(L\otimes_KF)|,\Lambda)\cong\holim_FC_{\Sing}(\bigsqcup_{\Hom(L,F)}|*|,\Lambda)\cong\bigoplus_{\Hom(L,C)}\Lambda
$$
which is canonically isomorphic, as a Galois representation,  to $\varinjlim_F\Lambda(\Spec L)(F)$ hence the claim. 
\end{proof}

We recall that an object $X$ of a triangulated category is compact if $\Hom(X,-)$ commutes with direct sums. 

\begin{prop}\label{redtocp}
	Let $F\colon\catT\ra\catT'$ and $G\colon\catT'\ra\catT$ be triangulated functors  commuting with direct sums between  triangulated categories $\catT$ and $\catT'$ generated  (as triangulated categories with small sums) by a set of compact objects $\mathcal{K}$ and $\mathcal{K}'$ respectively. Suppose that $F(X)$ is compact for each $X\in\mathcal{K}$ and that there is an invertible transformation $F\circ G\cong\id$. In order to prove that $F$ is a left adjoint to $G$ it suffices to prove \begin{equation}\label{adj}\Hom(X[n],GY)\cong\Hom(FX[n],Y)\end{equation}
	where $n$ varies in $\Z$ and where $X$ and $Y$ vary in $\mathcal{K}$ and $\mathcal{K}'$ respectively.
\end{prop}

\begin{proof}
	The invertible transformation gives rise to  a bi-functorial map
	$$\Hom(X,GY)\ra \Hom(FX,FGY) \cong \Hom(FX,Y).
	$$
	We want to show it is invertible for all $X$ and $Y$ by knowing it is invertible for  a set of compact generators of the two categories, and their shifts.
	
	Fix an object $X$ in the chosen class of compact generators $\mathcal{K}$ and let $\cat$ be the full subcategory of $\catT'$ whose objects $Y$ are such that \eqref{adj} is invertible for all $n$. Let $Y_1$ and $Y_2$ be in $\cat$ and pick a distinguished triangle 
	$$
	Y_1\ra Y_2\ra C\ra
	$$
	By the map of long exact sequences
	$$
	\resizebox{\textwidth}{!}{
		\xymatrix{
			\Hom(X,GY_1)\ar[d]^{\sim}\ar[r]&\Hom(X,GY_2)\ar[d]^{\sim}\ar[r]&\Hom(X,GC)\ar[d]\ar[r]&\Hom(X,GY_1[1])\ar[d]^{\sim}\ar[r]&\\
			\Hom(FX,Y_1)\ar[r]&\Hom(FX,Y_1)\ar[r]&\Hom(FX,C)\ar[r]&\Hom(FX,Y_1[1])\ar[r]&
		}
	}
	$$
	we deduce that $C$ is also in $\cat$. Let now $\{Y_i\}_{i\in I}$ be a class of objects in $\cat$. As $F$ maps compact objects to compact objects, and both $F$ and $G$ commute with direct sums we deduce:
	$$
	\begin{aligned}
	\Hom(X,G\bigoplus Y_i )&\cong \Hom(X,\bigoplus G Y_i )\cong \bigoplus \Hom(X,G Y_i )\\ &\cong
	\bigoplus \Hom(F X, Y_i )\cong \Hom(F X, \bigoplus Y_i ).
	\end{aligned}$$
	We have then showed that $\cat$ is closed both under direct sums and under cones, and it contains a family of generators for $\catT'$ and hence it coincides with it.
	
	We have then showed that for a class of compact generators $X$, the functor $Y\mapsto\Hom(X,GY)$ is corepresentable by $F X$. It suffices to invoke \cite[Lemma 5.6]{vezz-fw} to conclude.
\end{proof}

\begin{prop}\label{1isloc}
Suppose $\Q\subset\Lambda$. The object $\Lambda[n]$ is $\B^1$-local (even $\Frob$-$\B^1$-local if $K$ is perfect) in $\Ch\Sh_{\et}(\RigSm/K,\Lambda)$ and for any motive $\Lambda(X)$ of a smooth rigid analytic variety $X$, we have $\Hom(\Lambda(X),\Lambda[n])\cong H^n_{\Sing}(|X|_{\Berk},\Lambda)$.
\end{prop}

\begin{proof}
The fact that $H^n_{\et}(X,\Lambda)=H^n(|X|,\Lambda)$  follows from \cite[Remark 4.2.6-1]{dJ-vdP}. By overconvergence \cite[Proposition 8.2.6]{huber} we obtain $H^n(|X|,\Lambda)\cong  H^n(|X|_{\Berk},\Lambda)$ which coincides with its singular cohomology (see Proposition \ref{berkH}\eqref{b1}). We already proved the homotopy invariance of singular cohomology in Proposition \ref{maintop}. It is also Frobenius-invariant as the Frobenius induces a homeomorphism on $|X|_{\Berk}$.
\end{proof}

We are finally ready to prove Theorem \ref{mainadj}.

\begin{proof}[Proof of Theorem \ref{mainadj}]
 The functors $\LL\iota^* $ and $\LL B_{\Gal(K)}^*$ send compact objects to compact objects and commute with direct sums. By means of Propositions \ref{unit}, \ref{redtocp} and \cite[Theorem 1.2.34]{ayoub-rig}, it suffices to show that 
\begin{equation}\label{adj2}\Hom(\Lambda(X)[n],\LL\iota^*\Lambda(Y))\cong\Hom(\LL B^*_{\Gal(K)}\Lambda(X)[n],\Lambda(Y))\end{equation}
whenever $X$ is a connected, smooth quasi-compact rigid analytic variety 
and $Y=\Spa K'$ is Galois over $\Spa K$. We can consider the following Quillen adjunction (extending to $\Frobet$-motives too)
\begin{equation}\label{LeRe}
\adj{\LL e^*_{K'/K}}{\RigDA^{\eff}_{\et}(K,\Lambda)}{\RigDA^{\eff}_{\et}(K',\Lambda)}{\RR e_{K'/K*}}
\end{equation}
arising from the base change functor $\RigSm/K\ra\RigSm/K'$. From the equivalences
$$\begin{aligned}
\Hom(\Lambda(X)[n],\LL\iota^*\Lambda(Y))&\cong \Hom(\LL e^*_{K'/K}\Lambda(X)[n],\LL e^*_{K'/K}\LL\iota^*\Lambda(Y))^{\Gal(K'/K)}\\&\cong \Hom(\LL e^*_{K'/K}\Lambda(X)[n],\LL\iota^*\LL e^*_{K'/K}\Lambda(Y))^{\Gal(K'/K)}
\end{aligned}
$$ 
and
$$
\begin{aligned}
\Hom(\LL B^*_{\Gal(K)}\Lambda(X)[n],\Lambda(Y))&\cong \Hom(\LL e^*_{K'/K}\LL B^*_{\Gal(K)}\Lambda(X)[n],\LL e^*_{K'/K}\Lambda(Y))^{\Gal(K'/K)}\\&\cong \Hom(\LL B^*_{\Gal(K')}\LL e^*_{K'/K}\Lambda(X)[n],\LL e^*_{K'/K}\Lambda(Y))^{\Gal(K'/K)}
\end{aligned}
$$
we then deduce that we can prove \eqref{adj2} up to a finite Galois extension of the base field. In particular, we can assume 
 that $Y\cong\Lambda^{\oplus N}$ or even $Y\cong\Lambda$. 

We first remark that by Proposition  \ref{berkH}  and Proposition \ref{isXC} we have $$\Hom(\LL B_{\Gal(K)}^*\Lambda(X)[n],\Lambda)\cong  H^n_{\Sing}(|X_C|_{\Berk},\Lambda)^{\Gal(K^{\sep}/K)}\cong H^n_{\Sing}(|X|_{\Berk},\Lambda). $$  This also coincides with $\Hom(X[n],\Lambda)$ by means of  
 Proposition \ref{1isloc}, proving the statement.
\end{proof}

\begin{rmk}We now suppose that $K$ is perfect and we let $k$ be its residue field. 
	In the algebraic  context, the adjunction of Theorem \ref{main0} is studied  in \cite[Section 2.3]{ayoub-bv} and \cite[Section 2.2]{ayoub-zucker}. 
	We can show that the functor of Remark \ref{resgen} (in its version with transfers) is compatible with the functors $\omega_0$ in the sense that the following square is commutative:
	$$\xymatrix{
	\DM^{\eff}_{\et}(k,\Lambda)\ar[r]\ar[d]^{\omega_0}&\RigDM^{\eff}(K,\Lambda)\cong\RigDA^{\eff}_{\Frobet}(K,\Lambda)\ar[d]^{\omega_0}\\
	\catD_{\et}(k,\Lambda)\ar[r]&\catD_{\et}(K,\Lambda)
}
	$$
	In order to prove this, we can alternatively check the compatibility of the right adjoint functors. We recall that the functor on the top side is defined by means of the special fiber functor (inducing an equivalence) and the generic fiber functor defined on motives of formal schemes. Arguing like in the  proof of Theorem \ref{mainadj}, it suffices to show that for a geometrically connected, quasi-compact smooth formal scheme $\mfX/\mcO_K$ and any $n$, the complex $\Hom_\bullet(\Lambda(\mfX_{\eta}),\Lambda)$ is quasi-isomorphic to $\Hom_\bullet(\Lambda(\mfX_k),\Lambda)$. The former is quasi-isomorphic to $\Lambda[0]$ as shown in the previous proof (rigid varieties of good reduction are contractible by \cite[Section 5]{berk-contr}). The same holds for the latter, as shown in \cite[Corollary 4.2]{mvw}.
\end{rmk}

\begin{rmk}
Since $\LL B^*_{\Gal(K)}$ descends to the $\Frobet$-localization, we deduce from the adjunction above that the objects $\iota^*M$ are $\Frob$-local and hence ${\RigDA^{\eff}_{\Frobet}(K,\Lambda)_0}\cong{\RigDA^{\eff}_{\et}(K,\Lambda)_0}$. We will refer unambiguously to this category with ${\RigDA^{\eff}(K,\Lambda)_0}$.
\end{rmk}

\begin{rmk}\label{Bla}
As a corollary of Theorem \ref{main0}, we also obtain  that if $\Q\subset\Lambda$, the functor $\LL B^*$ of \ref{mainQ} is a left adjoint to the canonical functor $\catD(\Lambda)\ra\RigDA_{\et}^{\eff}(K,\Lambda)$ given by $\Lambda\mapsto\Lambda(K)$. Indeed, by comparing the right adjoint functors on the two sides, it suffices to check that $\LL(-)^*_{\Gal(K)}\circ\LL B^*_{\Gal(K)}\cong\LL B^* $ which follows from the isomorphism $ H^i(|X_L|_{\Berk},\Lambda)^{\Gal(L/K)}\cong H^i(|X|_{\Berk},\Lambda)$ of Proposition \ref{berkH} 
\end{rmk}

\begin{rmk}\label{rmkKperf}
Let $K^{\perf}$ be the completed perfection of $K$. We remark that the restriction of the adjunction in \eqref{LeRe} to Artin motives is an equivalence, since the two fields have the same Galois group. We remark also that $\LL e^*_{K^{\perf}/K}\circ \LL B_{\Gal(K)}^*\cong  \LL B_{\Gal(K^{\perf})}^*\circ \LL e^*_{K^{\perf}/K}$. Indeed, the two spaces $|X_K|_{\Berk}$ and $|X_{K^{\perf}}|_{\Berk}$ are actually homeomorphic, by \cite[Proposition 1.3.5(ii)]{berkovich}. 
\end{rmk}

We recall once more that an object $X$ of a triangulated category is compact if $\Hom(X,-)$ commutes with direct sums. Examples of compact objects in $\RigDA^{\eff}_{\Frobet}(K,\Lambda)$ are motives of quasi-compact smooth rigid analytic varieties over $K$ (see \cite[Proposition 1.2.34]{ayoub-rig}) and motives attached to the analytification of smooth algebraic varieties over $K$ (they are dualizable objects in the stable motivic category, by \cite[Lemma 1.3.29 and Lemma 2.5.30]{ayoub-rig} hence compact. They are also compact in the effective category by the Cancellation Theorem \cite[Corollary 2.5.49]{ayoub-rig}). The full subcategory of compact objects in a category $\catT$ will be denoted by $\catT^{\cp}$.

\begin{prop}
The adjunction of Theorem \ref{main0} restricts to compact objects defining   a left adjoint functor
$$
{\omega_0}\colon{\RigDA^{\eff}_{\et}(K,\Lambda)^{\cp}}\ra{\RigDA^{\eff}(K,\Lambda)_0^{\cp}}
$$
to the inclusion functor. If $K$ is perfect, the same is true for the adjunction defined on ${\RigDA^{\eff}_{\Frobet}(K,\Lambda)}$.

\end{prop}

\begin{proof}
It suffices to show that the functors $\LL\iota^*$ and $\LL B_{\Gal(K)}^*$ send a set of compact generators of the two categories to compact objects. For $\LL\iota^*$ this is immediate. For $\LL B_{\Gal(K)}^*$ this follows from Propositions \ref{berkH} and \ref{isXC}.
\end{proof}

\begin{rmk}
The functors $\LL B^*_{\Gal(K)}$ and $\omega_0$ defined above are tensorial, with respect to the monoidal structure on rigid analytic motives (see \cite[Propositions 4.2.76 and 4.4.63]{ayoub-th2}). Indeed, it suffices to check that for two rigid analytic varieties $X$ and $Y$ over $K$ the singular complex with $\Q$-coefficients  $C_{\Sing}(|X\times Y|_C)$ is quasi-isomorphic to $C_{\Sing}(|X|_C)\otimes C_{\Sing}(|Y|_C)$. This follows from \cite[Corollary 8.7]{berk-contr} and the usual K\"unneth formula for singular homology.
\end{rmk}

\begin{rmk}The fact that we are dealing with the category $\RigDA_{\et}^{\eff}(K,\Lambda)$ (and not simply with $\RigSH^{\eff}_{\et}(K)$) and the hypothesis $\Q\subset\Lambda$ are  used in this section several times: for example, in order to deduce properties of the  functor $L\mapsto C_{\Sing}(|X_L|,\Lambda)$ (related to homology) out of the properties of singular \emph{co}-homology of Berkovich spaces (see Definition \ref{defBGal}) as well as to invoke the result of \cite{dJ-vdP} in Proposition \ref{1isloc}.
\end{rmk}

\section{Compatibility with the tilting equivalence}\label{tilt}

Suppose now that $K$ is a perfectoid field of characteristic $0$ (that is, a complete valued field of mixed characteristic $(0,p)$ endowed with a non-discrete valuation, such that Frobenius is surjective on $\mcO_K/p$ see \cite[Definition 3.1]{scholze}) and $\Q\subset\Lambda$. Under such hypotheses, we can define a perfect  complete valued field $K^\flat$ of positive characteristic (the tilt of $K$) and construct a "motivic tilting equivalence" (see \cite{vezz-fw}):
\[
\RigDA^{\eff}_{\et}(K,\Lambda)\cong\PerfDA^{\eff}_{\et}(K,\Lambda)\cong\PerfDA^{\eff}_{\et}(K^\flat,\Lambda)\cong\RigDA^{\eff}_{\Frobet}(K^\flat,\Lambda)
\]
which is a obtained by ``descending'' Scholze's tilting equivalence between perfectoid spaces over $K$ and $K^\flat$ (see \cite[Proposition 6.17]{scholze}). 

On the other hand, the category $\catD_{\et}(K,\Lambda)\cong\RigDA^{\eff}(K,\Lambda)_0$ is equivalent to $\catD_{\et}(K^\flat,\Lambda)\cong\RigDA^{\eff}(K^\flat,\Lambda)_0$ by means of the functor that associates to a (perfectoid) finite \'etale extension $L/K$ the extension $L^\flat/K^\flat$: indeed Scholze's tilting equivalence restricts to an equivalence over the finite \'etale extensions of $K$ and $K^\flat$ (this is the classic theorem of Fontaine and Wintenberger). We now specify that the two equivalences above are compatible with each other, and also to the Berkovich realization defined above. %

\begin{prop}
Let $K$ be a perfectoid field  and let $\Lambda$ be a $\Q$-algebra. The functor $\omega_0$ commutes with the tilting equivalence.
\end{prop}

\begin{proof}
By means of the adjunction property, we can alternatively prove that the following diagram is commutative
$$
\xymatrix{
\catD_{\et}(K,\Lambda)\ar[r]^-{\LL\iota^*} \ar@{<->}[d]^{\sim} &
		\RigDA^{\eff}_{\et}(K,\Lambda)\ar@{<->}[d]^{\sim}\\
	\catD_{\et}(K^\flat,\Lambda)\ar[r]^-{\LL\iota^*} &
	\RigDA^{\eff}_{\Frobet}(K^\flat,\Lambda)
	}
$$
for a perfectoid field $K$ of characteristic zero with tilt $K^\flat$.

We will now decompose this diagram in some sub-squares following the picture of \cite[Page 40]{vezz-fw}. We recall (see \cite[Theorem 7.11]{vezz-fw}) that the equivalence $\PerfDA^{\eff}_{\et}(K,\Lambda)\cong\RigDA^{\eff}_{\et}(K,\Lambda)$ is obtained as the composite of the two functors
$$
\PerfDA^{\eff}_{\et}(K,\Lambda)\stackrel{\LL j^*}{\ra}\wRigDA_{\widehat{\B}^1}^{\eff}(K,\Lambda)\stackrel{\LL i_!}{\ra}\RigDA^{\eff}_{\et}(K,\Lambda)
$$
where the category in the middle is the category of semi-perfectoid motives (denoted by $\wwRigDA_{\widehat{\B}^1}^{\eff}(K,\Lambda)$ in \cite[Definition 3.22]{vezz-fw}) the functor $\LL j^*$ is induced by the inclusion of smooth perfectoid spaces inside smooth semi-perfectoid spaces, while $\LL i_!$ is the left adjoint of the functor $\LL i^*$  induced by the inclusion of smooth rigid analytic varieties inside smooth semi-perfectoid spaces.

 First, we consider the diagram
$$
\xymatrix{
	&	\RigDA^{\eff}_{\et}(K,\Lambda)\\
	\catD_{\et}(K,\Lambda)\ar[ur]^-{\LL\iota_1^*}\ar[r]^-{\LL\iota_2^*} \ar[dr]_-{\LL\iota_3^*}  &
		\wRigDA^{\eff}_{\widehat{\B}^1}(K,\Lambda)\ar[u]_{\LL i_!}\\
	 &
	\PerfDA^{\eff}_{\et}(K,\Lambda)\ar[u]_{\LL j^*}\ar@/_4pc/[uu]_{\sim}
}
$$
where we indicate with $\iota_1,\iota_2,\iota_3$ the inclusion of the small \'etale site over $K$ in the big \'etale site of rigid analytic varieties resp. smooth semi-perfectoid spaces resp. smooth perfectoid spaces. The lower square commutes by the equivalence $\iota_2\cong j\circ \iota_3$. Similarly, we have an equivalence $\iota_2\cong  i\circ\iota_1$ which implies $\LL\iota_2^*\cong\LL i^*\circ\LL\iota_1^*$. Since $\LL i_!\circ\LL i^*$ is equivalent to the identity by \cite[Theorem 5.5]{vezz-fw}, this yields $\LL\iota_1^*\cong\LL i_!\circ\LL\iota^*_2$ hence the commutativity of the upper triangle.

We now consider the following square (see \cite[Proposition 3.23]{vezz-fw})
$$
\xymatrix{
	\catD_{\et}(K,\Lambda)\ar[r]^-{\LL\iota^*} \ar@{<->}[d]^{\sim} &
	\PerfDA^{\eff}_{\et}(K,\Lambda)\ar@{<->}[d]^{\sim}\\
	\catD_{\et}(K^\flat,\Lambda)\ar[r]^-{\LL\iota^*} &
	\PerfDA^{\eff}_{\et}(K^\flat,\Lambda)
}
$$
which commutes by definition of the tilting equivalence on both sides.

We are left to consider the triangle (see \cite[Theorem 6.9]{vezz-fw})
$$
\xymatrix{
	&	\RigDA^{\eff}_{\Frobet}(K^\flat,\Lambda)\ar[dd]^{\LL \Perf^*}_{\sim}\\
	\catD_{\et}(K^\flat,\Lambda)\ar[ur]^-{\LL\iota_1^*} \ar[dr]_-{\LL\iota_3^*} \\
	&
	\PerfDA^{\eff}_{\et}(K^\flat,\Lambda)
}
$$
where now the equivalence on the right is induced simply by means of the (completed) perfection functor $\Perf$. It is then immediate to prove it commutes (a finite \'etale extension of $K^\flat$ is already perfect).
	\end{proof}

\section{Compatibility with the \'etale realization}\label{etale}

We show in this section that our main theorem in Section \ref{artq} can be interpreted as a motivic version of the results of Berkovich \cite{berk-tate} showing that the singular cohomology $H^*_{\Sing}(|X^{\an}_C|_{\Berk},\Q_\ell)$ of the Berkovich space associated to the analytification of an algebraic variety over $K$ are canonically isomorphic to the weight-zero part of the \'etale cohomology $H^*_{\et}(X_C,\Q_\ell)$ for $\ell\neq p$. In particular, we show how to obtain these equivalence via our theorem and the \'etale realization. This allows us to generalize them further to arbitrary analytic varieties.

From now on, we assume that the residue field of $K$ is finite of characteristic $p$ and we pick a prime $\ell\neq p$. The functors that we will consider are insensitive to  base change over the completed perfection of $K$ (see the remark in \cite[Proposition 2.3.7]{huber} and Remark \ref{rmkKperf}). We will then assume for simplicity that $K$ is perfect.

We recall here the basic properties of the $\ell$-adic realization functor for rigid analytic motives, constructed in \cite[Section 3.1]{bamb-vez} (see also \cite[Example 2.23]{ayoub-nr}).

\begin{prop}\label{Rell}
Fix a prime  $\ell$ coprime to the residue characteristic $p$ of $K$. There is a triangulated monoidal functor $$
\mathfrak{R}_{\et,\ell}\colon\RigDA^{\eff}_{\Frobet}(K,\Q)^{\cp}\ra{\widehat{\catD}}_{\et}^{\cp}(K,\Q_\ell)
$$
where the category ${\widehat{\catD}}_{\et}^{\cp}(K,\Q_\ell)$ is the derived category of constructible $\ell$-adic  sheaves following Ekedhal (see \cite[Definition 9.3]{ayoub-etale} and \cite[Section 5.5]{bs}). It has the following properties:\begin{enumerate}
	\item $\mathfrak{R}_{\et,\ell}$ is tensorial and triangulated.
		\item For any smooth rigid analytic variety $X$, the Galois representation attached to $(H_i\mathfrak{R}_{\et,\ell}(\Q_\ell(X)))^\vee$ is the \'etale representation $H^i_{\et}(X_C,\Q_\ell)$.
	\item\label{prop:3rell} The composition $\mathfrak{R}_{\et,\ell}\circ\LL\iota^*$ is canonically isomorphic to the  functor $\nu^*\colon\catD_{\et}(K,\Q)^{\cp}\ra{\widehat{\catD}}_{\et}^{\cp}(K,\Q_\ell)$ induced by extending coefficients.
\end{enumerate}
\end{prop}
\begin{proof}
The main statement and the first property follow from \cite[Theorem 3.2]{bamb-vez}. The second property is proved in \cite[Remark 3.3]{bamb-vez}. The third property can be proved at an integral level, be inspecting the functor $\catD_{\et}(K,\Z)^{\cp}\ra{\widehat{\catD}}_{\et}^{\cp}(K,\Z_\ell)$ induced by the integral version of $\mathfrak{R}_{\et,\ell}$ (see \cite[Theorem 3.2]{bamb-vez}). By its construction, based on the  Rigidity Theorem \cite[Theorem 2.1]{bamb-vez}, we see that it is canonically equivalent to the functor induced by extending coefficients, as wanted.
\end{proof}

If we want to relate the functor $\mathfrak{R}_{\et,\ell}$ with Berkovich's version of Tate's conjecture \cite{berk-tate} we need to introduce weights of Weil numbers appearing as eigenvalues of a lift of Frobenius. We then 
 consider the functor $H_*\colon{\widehat{\catD}}_{\et}^{\cp}(K,\Q_\ell)\ra\bigoplus\Rep_{\ct}(\Gal(K),\Q_\ell)$ associating to a complex its homology sheaves, which are  $\Q_\ell$ vector spaces endowed with a continuous action of $\Gal(K)$ (with respect to  the $\ell$-adic topology on $\Q_\ell$). We use the following notation of Berkovich.

\begin{dfn}
Let $V$ be a continuous $\ell$-adic representation of $\hat{\Z}$ and let $F$ be a topological generator of $\hat{\Z}$. We say $V$ has \emph{weight zero} if  the eigenvalues of $F$ are Weil numbers of weight equal to $0$. The subcategory of representations $\Rep_{\ct}(\hat{\Z},\Q_\ell)$ they form will be denoted by $ \Rep_{\ct}(\hat{\Z},\Q_\ell)_0$. For any representation $V$ we let $V_0$ [resp. $V^0$] be the maximal sub-representation [resp. quotient representation] of $V$ such that the eigenvalues of $F$ are Weil numbers of weight equal to $0$. %
Since the inverse of a Weil number of weight $0$ is again a Weil number of weight $0$, one has $V^0\cong ((V^\vee)_0)^\vee$.
\end{dfn}

\begin{dfn}
We let $\tilde{\omega}_0$ be the functor $\Rep_{\ct}(\hat{\Z},\Q_\ell)^{\cp}\ra \Rep_{\ct}(\hat{\Z},\Q_\ell)^{\cp}_0$ mapping $V$ to $V^0$. It is a left adjoint functor to the canonical inclusion.
\end{dfn}

If $V$ be a continuous $\ell$-adic Galois representation, we can consider  $F\in\Gal(K)$ to be a lift of the geometric Frobenius and restrict $V$ to a representation of $\langle F\rangle\cong\hat{\Z}$. This defines a functor 
$$
\bigoplus\Rep_{\cont}(\Gal(K),\Q_\ell)\ra \bigoplus\Rep_{\cont}(\hat{\Z},\Q_\ell).
$$

By composition, we have then constructed a functor (depending on the choice of $F$)
$$
H_*^F\mathfrak{R}_{\et,\ell}\colon\RigDA^{\eff}_{\Frobet}(K,\Q_\ell)^{\cp}\ra \bigoplus\Rep_{\cont}(\hat{\Z},\Q_\ell)^{\cp}
$$
which obviously restricts to a functor
$$
H_*^F\mathfrak{R}_{\et,\ell}\colon\RigDA^{\eff}_{\Frobet}(K,\Q_\ell)_0^{\cp}\ra \bigoplus\Rep_{\cont}(\hat{\Z},\Q_\ell)^{\cp}_0
$$
since  the Galois action on come compact Artin motive factors over a finite quotient of the Galois group. We now show that the functors $\omega_0$'s are compatible with the two functors above.

\begin{prop}\label{comm}
The following diagram is commutative:
$$
\xymatrix{
	\RigDA^{\eff}_{\Frobet}(K,\Q_\ell)^{\cp}\ar[r]^-{H_*^F\mathfrak{R}_{\et,\ell}}\ar[d]^{\omega_0}&
	\bigoplus\Rep_{\cont}(\hat{\Z},\Q_\ell)^{\cp}\ar[d]^{\tilde{\omega}_0}\\
		\RigDA^{\eff}_{\Frobet}(K,\Q_\ell)^{\cp}_0\ar[r]^-{H_*^F\mathfrak{R}_{\et,\ell}}&
		\bigoplus	\Rep_{\cont}(\hat{\Z},\Q_\ell)_0^{\cp}\\
	}
$$
\end{prop}

\begin{proof}
We denote by $\iota$ be the right adjoint functors of $\omega_0$ and $\tilde{\omega}_0$ (the obvious inclusions). From the commutativity $\iota\circ H_*^F\mathfrak{R}_{\et,\ell}\cong H_*^F\mathfrak{R}_{\et,\ell}\circ\iota$ and the unit of the adjunction we deduce the existence of a natural transformation
$$
H_*^F\mathfrak{R}_{\et,\ell}\Rightarrow H_*^F\mathfrak{R}_{\et,\ell}\circ \iota\circ \omega_0\cong \iota \circ H_*^F\mathfrak{R}_{\et,\ell} \circ \omega_0
$$
which induces a natural transformation $\eta\colon\tilde{\omega}_0\circ H_*^F\mathfrak{R}_{\et,\ell}\Rightarrow H_*^F\mathfrak{R}_{\et,\ell}\circ\omega_0$.  We let $\catT$ be the full subcategory of $	\RigDA^{\eff}_{\Frobet}(K,\Q_\ell)^{\cp}$ of those objects $C$ such that $\eta(C)$ is invertible. All the functors involved commute with finite sums and shifts, so we deduce that $\catT$ is closed under these operations. We now let 
$$
X\ra Y\ra C\ra
$$
be a distinguished triangle of $	\RigDA^{\eff}_{\Frobet}(K,\Q_\ell)^{\cp}$ with $X$ and $Y$ inside $\catT$. Since the functors $\omega_0$ and $\mathfrak{R}_{\et,\ell}$ are triangulated, we obtain the following long exact sequence
$$
H_i(\mathfrak{R}_{\et,\ell}\omega_0(X))\ra H_i(\mathfrak{R}_{\et,\ell}\omega_0(Y))\ra H_i(\mathfrak{R}_{\et,\ell}\omega_0(C))\ra H_{i-1}(\mathfrak{R}_{\et,\ell}\omega_0(X))
$$
On the other hand, since $\mathfrak{R}_{\et,\ell}$ is triangulated and the functors $V\mapsto V^\vee$, $V\mapsto V_0$ are exact, we also deduce the following long exact sequence 
$$
\tilde{\omega}_0H_i(\mathfrak{R}_{\et,\ell}(X)))\ra \tilde{\omega}_0H_i(\mathfrak{R}_{\et,\ell}(Y))\ra \tilde{\omega}_0H_i(\mathfrak{R}_{\et,\ell}(C))\ra \tilde{\omega}_0H_{i-1}(\mathfrak{R}_{\et,\ell}(X))
$$
The transformation $\eta$ induces a morphism between the two long exact sequences above. By the five-lemma and the isomorphisms $H_i(\mathfrak{R}_{\et,\ell}\omega_0(X))\cong\tilde{\omega}_0 H_i(\mathfrak{R}_{\et,\ell}(X))$ and $H_i(\mathfrak{R}_{\et,\ell}\omega_0(Y))\cong\tilde{\omega}_0 H_i(\mathfrak{R}_{\et,\ell}(Y))$ we then deduce $H_i(\mathfrak{R}_{\et,\ell}\omega_0(C))\cong\tilde{\omega}_0 H_i(\mathfrak{R}_{\et,\ell}(C))$ proving that $\eta(C)$ is invertible as well. We have therefore proved that $\catT$ is closed under cones.

In order to show $\catT=	\RigDA^{\eff}_{\Frobet}(K,\Q_\ell)^{\cp}$ it then suffices to prove that a set of generators of $	\RigDA^{\eff}_{\Frobet}(K,\Q_\ell)^{\cp}$ (as a triangulated category) lie in $\catT$. For example, we can take motives of the form $\Q_\ell(X)$ for $X$ a rigid analytic variety of potentially good reduction by \cite[Theorem 2.5.34]{ayoub-rig} and \cite{vezz-DADM}. We  then fix $X$ and we suppose that for some finite extension $K'/K$ there is a smooth formal model $\mfX$ over $\mcO_{K'}$ whose generic fiber is $X_{K'}$. In particular, $|X_{K'}|_{\Berk}$ is weakly contractible (see \cite[Section 5]{berk-contr}). We obtain that $$(H_*^F\mathfrak{R}_{\et,\ell}\circ\omega_0(\Lambda(X)))^\vee\cong\bigoplus H^i_{\Sing}(|X_C|_{\Berk},\Q_\ell)\cong H^0_{\Sing}(|X_C|_{\Berk},\Q_\ell)\cong\bigoplus_{\pi_0(X_C)}\Q_\ell$$
where the first isomorphism follows from the definition  $\omega_0\cong\iota\circ\LL B^*_{\Gal(K)}$, Proposition \ref{isXC} and Proposition \ref{Rell}. 

 By \cite[Corollary 4.5(iii) and Corollary 5.4]{berk-van} if we let $\bar{k}$ be the residue field of $C$, we obtain on the other hand that $H^i_{\et}(X_C,\Q_\ell)\cong H^i_{\et}(\mfX_{\sigma\bar{k}},\Q_\ell)$ whose weight-zero part is zero unless $i=0$, and equal to $\bigoplus_{\pi_0(\mfX_{\sigma\bar{k}})}\Q_\ell$ otherwise (from the Weil conjecture proved by Deligne \cite{del-WI}). Also, since $\pi_0(X_C)\cong\pi_0(\mfX_{\sigma\bar{k}})$ as $\Gal(K)$-sets, we  deduce $H^i_{\et}(X_C,\Q_\ell)_0\cong (H_*^F\mathfrak{R}_{\et,\ell}\circ\omega_0(\Lambda(X)))^\vee$ which entails $\tilde{\omega}_0\mathfrak{R}_{\et,\ell}\Q_\ell(X)\cong \mathfrak{R}_{\et,\ell}\omega_0\Q_\ell(X)$ as wanted.
\end{proof}

We can finally generalize Berkovich's formulas \cite{berk-tate} to arbitrary rigid analytic compact motives.%

\begin{cor}\label{maincor}Let $M$ be in $\RigDA^{\eff}_{\Frobet}(K,\Q)^{\ct}$ and $i$ be in $\Z$. Then $H^i(\LL B^*_{\Gal(K)}M)\otimes\Q_{\ell}$ coincides with $H^i_{\et}(M_C,\Q_{\ell})_0$. In particular, 
if $X$ is a smooth quasi-compact rigid variety or an analytification of an algebraic variety, we have $H^i_{\Sing}(|X_C|_{\Berk},\Q_\ell)\cong H^i_{\et}(X_C,\Q_\ell)_0$.
\end{cor}

\begin{proof}
	By the definition of $\omega_0$ and Proposition \ref{Rell}, the two groups of the statement coincide precisely with the $i$-th cohomology groups of $ (H_*^F\mathfrak{R}_{\et,\ell} \circ \omega_0)(M)$ and $    (\widetilde{\omega}_0\circ H_*^F\mathfrak{R}_{\et,\ell})(M)$ respectively. The first part of the corollary then follows from Proposition \ref{comm}. 
	
	For the second part, it suffices to take $M=\Lambda(X)$ and  refer to the proof of the previous proposition where we showed that $ \bigoplus H^i_{\Sing}(|X_C|_{\Berk},\Q_\ell)^\vee\cong (H_*^F\mathfrak{R}_{\et,\ell}\circ\omega_0)(\Lambda(X))$ and $\bigoplus (H^i_{\et}(X_C,\Q_\ell)_0)^\vee\cong (\omega_0\circ H_*^F\mathfrak{R}_{\et,\ell})(\Lambda(X))$.
\end{proof}

\begin{cor}Let $M$ be in $\RigDA^{\eff}_{\Frobet}(K,\Q)^{\ct}$ and $i$ be in $\Z$. Then $H^i(\LL B^*M)\otimes\Q_{\ell}$ coincides with $H^i_{\et}(M_C,\Q_{\ell})^{\Gal(K^{\sep}/K)}$. In particular, 
	if $X$ is a smooth quasi-compact rigid variety or an analytification of an algebraic variety, we have $H^i_{\Sing}(|X|_{\Berk},\Q_\ell)\cong H^i_{\et}(X_C,\Q_\ell)^{\Gal(K^{\sep}/K)}$.
\end{cor}

\begin{proof}
	It suffices to argue like in \cite[Corollary 1.2]{berk-tate}. Alternatively, one can use Remark \ref{Bla} and the same strategy of the proof above.
\end{proof}

\begin{rmk}
We also obtain the versions of Berkovich's formulas for cohomology \emph{with compact support}
$$
H^i_{\Sing,c}(|X_C|_{\Berk},\Q_\ell)\cong H^i_{\et,c}(X_C,\Q_\ell)_0
$$
for any algebraic variety $X$ over $K$. It suffices to apply Corollary \ref{maincor} to the analytification (see \cite[Proposition 1.3.6]{ayoub-rig}) of the motive $M^c(X)$ computing cohomology with compact support, following the notation of \cite[Chapter 16]{mvw}.%
\end{rmk}

\begin{rmk}
Berkovich's results on the comparison of the weight-zero part of cohomology and singular cohomology of the Berkovich space go beyond the formula that we generalize here in Corollary \ref{maincor}. Indeed, there are versions of it for trivially valued finitely generated fields and archimedean fields (see the cases (b) and (c) of \cite{berk-tate}) as well as for $p$-adic cohomologies (see case (a'') of \cite{berk-tate}).  In this work, we heavily relied on the results of \cite{berk-contr} which are proved there only for non-trivially valued non-archimedean fields, hence our Assumption \ref{assu}. Nonetheless, we believe that the other versions of Berkovich's formula have a motivic interpretation too. Such refinements are left for future work.
\end{rmk}

\section*{Acknowledgments}

We thank J\"org Wildeshaus for the numerous discussions  and for having suggested the content of the main theorem in Section \ref{artq}. We  thank Geoffroy Horel and Jacklyn Lang for their useful remarks on earlier versions of this text, and for many fruitful discussions on the content of this article. We also thank an anonymous referee for his/her useful suggestions and corrections.

  \bibliographystyle{plain}

\begin{thebibliography}{10}
  
  \bibitem{ayoub-etale}
  Joseph Ayoub.
  \newblock La r\'ealisation \'etale et les op\'erations de {G}rothendieck.
  \newblock Annales scientifiques de l'{E}cole normale sup\'erieure 47, fascicule
    1 (2014), 1-145.
  
  \bibitem{ayoub-nr}
  Joseph Ayoub.
  \newblock Nouvelles cohomologies de weil en caractéristique positive.
  \newblock Preprint.
  
  \bibitem{ayoub-th2}
  Joseph Ayoub.
  \newblock Les six op\'erations de {G}rothendieck et le formalisme des cycles
    \'evanescents dans le monde motivique, {II}.
  \newblock {\em Ast\'erisque}, (315):vi+364 pp. (2008), 2007.
  
  \bibitem{ayoub-note}
  Joseph Ayoub.
  \newblock Note sur les op\'erations de {G}rothendieck et la r\'ealisation de
    {B}etti.
  \newblock {\em J. Inst. Math. Jussieu}, 9(2):225--263, 2010.
  
  \bibitem{ayoub-icm}
  Joseph Ayoub.
  \newblock A guide to (\'{e}tale) motivic sheaves.
  \newblock In {\em Proceedings of the {I}nternational {C}ongress of
    {M}athematicians---{S}eoul 2014. {V}ol. {II}}, pages 1101--1124. Kyung Moon
    Sa, Seoul, 2014.
  
  \bibitem{ayoub-h1}
  Joseph Ayoub.
  \newblock L'alg\`ebre de {H}opf et le groupe de {G}alois motiviques d'un corps
    de caract\'eristique nulle, {I}.
  \newblock {\em J. Reine Angew. Math.}, 693:1--149, 2014.
  
  \bibitem{ayoub-h2}
  Joseph Ayoub.
  \newblock L'alg\`ebre de {H}opf et le groupe de {G}alois motiviques d'un corps
    de caract\'eristique nulle, {II}.
  \newblock {\em J. Reine Angew. Math.}, 693:151--226, 2014.
  
  \bibitem{ayoub-rig}
  Joseph Ayoub.
  \newblock Motifs des vari\'et\'es analytiques rigides.
  \newblock {\em M\'em. Soc. Math. Fr. (N.S.)}, (140-141):vi+386, 2015.
  
  \bibitem{ayoub-concon}
  Joseph Ayoub.
  \newblock Motives and algebraic cycles: a selection of conjectures and open
    questions.
  \newblock In {\em Hodge theory and {$L^2$}-analysis}, volume~39 of {\em Adv.
    Lect. Math. (ALM)}, pages 87--125. Int. Press, Somerville, MA, 2017.
  
  \bibitem{ayoub-bv}
  Joseph Ayoub and Luca Barbieri-Viale.
  \newblock 1-motivic sheaves and the {A}lbanese functor.
  \newblock {\em J. Pure Appl. Algebra}, 213(5):809--839, 2009.
  
  \bibitem{ayoub-zucker}
  Joseph Ayoub and Steven Zucker.
  \newblock Relative {A}rtin motives and the reductive {B}orel-{S}erre
    compactification of a locally symmetric variety.
  \newblock {\em Invent. Math.}, 188(2):277--427, 2012.
  
  \bibitem{bamb-vez}
  Federico Bambozzi and Alberto Vezzani.
  \newblock Rigidity for rigid analytic motives.
  \newblock arXiv:1810.04968 [math.AG], 2018.
  
  \bibitem{BVK}
  Luca Barbieri-Viale and Bruno Kahn.
  \newblock On the derived category of 1-motives.
  \newblock {\em Ast\'erisque}, (381):xi+254, 2016.
  
  \bibitem{berkovich}
  Vladimir~G. Berkovich.
  \newblock {\em Spectral theory and analytic geometry over non-{A}rchimedean
    fields}, volume~33 of {\em Mathematical Surveys and Monographs}.
  \newblock American Mathematical Society, Providence, RI, 1990.
  
  \bibitem{berk-van}
  Vladimir~G. Berkovich.
  \newblock Vanishing cycles for formal schemes.
  \newblock {\em Invent. Math.}, 115(3):539--571, 1994.
  
  \bibitem{berk-contr}
  Vladimir~G. Berkovich.
  \newblock Smooth {$p$}-adic analytic spaces are locally contractible.
  \newblock {\em Invent. Math.}, 137(1):1--84, 1999.
  
  \bibitem{berk-tate}
  Vladimir~G. Berkovich.
  \newblock An analog of {T}ate's conjecture over local and finitely generated
    fields.
  \newblock {\em Internat. Math. Res. Notices}, (13):665--680, 2000.
  
  \bibitem{bs}
  Bhargav Bhatt and Peter Scholze.
  \newblock The pro-\'etale topology for schemes.
  \newblock {\em Ast\'erisque}, (369):99--201, 2015.
  
  \bibitem{BGR}
  Siegfried Bosch, Ulrich G{\"u}ntzer, and Reinhold Remmert.
  \newblock {\em Non-{A}rchimedean analysis}, volume 261 of {\em Grundlehren der
    Mathematischen Wissenschaften [Fundamental Principles of Mathematical
    Sciences]}.
  \newblock Springer-Verlag, Berlin, 1984.
  \newblock A systematic approach to rigid analytic geometry.
  
  \bibitem{gall-chou}
  Utsav Choudhury and Martin Gallauer~Alves de~Souza.
  \newblock Homotopy theory of dg sheaves.
  \newblock {\em Comm. Algebra}, Published online,
    doi:10.1080/00927872.2018.1554744, 2019.
  
  \bibitem{cd}
  Denis-{C}harles Cisinski and Fr\'ed\'eric D\'eglise.
  \newblock Triangulated categories of mixed motives.
  \newblock arXiv:0912.2110v3 [math.AG], 2009.
  
  \bibitem{cd-etale}
  Denis-Charles Cisinski and Fr\'{e}d\'{e}ric D\'{e}glise.
  \newblock \'{E}tale motives.
  \newblock {\em Compos. Math.}, 152(3):556--666, 2016.
  
  \bibitem{dJ-vdP}
  Johan de~Jong and Marius van~der Put.
  \newblock \'etale cohomology of rigid analytic spaces.
  \newblock {\em Doc. Math.}, 1:No. 01, 1--56, 1996.
  
  \bibitem{del-WI}
  Pierre Deligne.
  \newblock La conjecture de {W}eil. {I}.
  \newblock {\em Inst. Hautes \'Etudes Sci. Publ. Math.}, (43):273--307, 1974.
  
  \bibitem{dugger}
  Daniel Dugger.
  \newblock Universal homotopy theories.
  \newblock {\em Adv. Math.}, 164(1):144--176, 2001.
  
  \bibitem{dhi}
  Daniel Dugger, Sharon Hollander, and Daniel~C. Isaksen.
  \newblock Hypercovers and simplicial presheaves.
  \newblock {\em Math. Proc. Cambridge Philos. Soc.}, 136(1):9--51, 2004.
  
  \bibitem{dug-isa}
  Daniel Dugger and Daniel~C. Isaksen.
  \newblock Topological hypercovers and {$\Bbb A^1$}-realizations.
  \newblock {\em Math. Z.}, 246(4):667--689, 2004.
  
  \bibitem{EKMM}
  A.~D. Elmendorf, I.~K\v r\'\i~\v z, M.~A. Mandell, and J.~P. May.
  \newblock Modern foundations for stable homotopy theory.
  \newblock In {\em Handbook of algebraic topology}, pages 213--253.
    North-Holland, Amsterdam, 1995.
  
  \bibitem{tohoku}
  Alexander Grothendieck.
  \newblock Sur quelques points d'alg\`ebre homologique.
  \newblock {\em T\^ohoku Math. J. (2)}, 9:119--221, 1957.
  
  \bibitem{hovey}
  Mark Hovey.
  \newblock {\em Model categories}, volume~63 of {\em Mathematical Surveys and
    Monographs}.
  \newblock American Mathematical Society, Providence, RI, 1999.
  
  \bibitem{huber2}
  Roland Huber.
  \newblock A generalization of formal schemes and rigid analytic varieties.
  \newblock {\em Math. Z.}, 217(4):513--551, 1994.
  
  \bibitem{huber}
  Roland Huber.
  \newblock {\em \'{E}tale cohomology of rigid analytic varieties and adic
    spaces}.
  \newblock Aspects of Mathematics, E30. Friedr. Vieweg \& Sohn, Braunschweig,
    1996.
  
  \bibitem{ivorra-h}
  Florian Ivorra.
  \newblock Perverse, {H}odge and motivic realizations of \'etale motives.
  \newblock {\em Compos. Math.}, 152(6):1237--1285, 2016.
  
  \bibitem{ivo-seb}
  Florian Ivorra and Julien Sebag.
  \newblock Artin motives, weights and motivic sheaves.
  \newblock {\em Michigan Math. J.}, Advance publication,
    doi:10.1307/mmj/1551258026, 2019.
  
  \bibitem{mvw}
  Carlo Mazza, Vladimir Voevodsky, and Charles Weibel.
  \newblock {\em Lecture notes on motivic cohomology}, volume~2 of {\em Clay
    Mathematics Monographs}.
  \newblock American Mathematical Society, Providence, RI, 2006.
  
  \bibitem{prasolov}
  Andrei~V. Prasolov.
  \newblock On the universal coefficients formula for shape homology.
  \newblock {\em Topology Appl.}, 160(14):1918--1956, 2013.
  
  \bibitem{quick}
  Gereon Quick.
  \newblock Continuous group actions on profinite spaces.
  \newblock {\em J. Pure Appl. Algebra}, 215(5):1024--1039, 2011.
  
  \bibitem{scholze}
  Peter Scholze.
  \newblock Perfectoid spaces.
  \newblock {\em Publ. Math. Inst. Hautes \'Etudes Sci.}, 116:245--313, 2012.
  
  \bibitem{vezz-DADM}
  Alberto Vezzani.
  \newblock Effective motives with and without transfers in characteristic $p$.
  \newblock {\em Adv. Math.}, 306:852--879, 2017.
  
  \bibitem{vezz-rigidreal}
  Alberto Vezzani.
  \newblock The {M}onsky-{W}ashnitzer and the overconvergent realizations.
  \newblock {\em Int. Math. Res. Not. IMRN}, (11):3443--3489, 2018.
  
  \bibitem{vezz-fw}
  Alberto Vezzani.
  \newblock A motivic version of the theorem of {F}ontaine and {W}intenberger.
  \newblock {\em Compos. Math.}, 155(1):38--88, 2019.
  
  \bibitem{vezz-tilt4rigid}
  Alberto Vezzani.
  \newblock Rigid cohomology via the tilting equivalence.
  \newblock {\em J. Pure Appl. Algebra}, 223(2):818--843, 2019.
  
  \end{thebibliography}

\end{document}